\documentclass[amscd,amssymb,verbatim,12pt]{amsart}
\usepackage{epsfig}
\usepackage{graphicx}
\newif\ifAMS
\IfFileExists{amssymb.sty}
  {\AMStrue\usepackage{amssymb}}
  {\usepackage{latexsym}}
  \usepackage{enumerate}
\theoremstyle{plain}

\newtheorem*{Co-area}{Length area inequality}
\newtheorem*{Loewner}{Loewner's inequality}

\newtheorem{Thm}{Theorem}[section]

\theoremstyle{definition}

\theoremstyle{remark}

\newcommand{\interior}{^{ \kern-5pt ^\circ}}

\newcommand {\al}{\alpha}

\begin{document}
\title
{Short loops in surfaces with a circle boundary component }

\author
{Panos Papasoglu }


\email {} \email {papazoglou@maths.ox.ac.uk}

\address
{Mathematical Institute,  University of Oxford, Andrew Wiles Building, Woodstock Rd, Oxford OX2 6GG, U.K. }

\address
{ }

\begin{abstract} It is a classical theorem of Loewner that the systole of a Riemannian torus can be bounded in terms of its area.
We answer a question of a similar flavor of Robert Young showing
that if $T$ is a Riemannian 2-torus with boundary in $\mathbb R ^n$, such that the boundary curve is a standard unit circle,
then the length of the shortest non-contractible loop in $T$ is bounded in terms of the area of $T$.
\end{abstract}
\maketitle
\section{Introduction}
Robert Young in \cite{Y} conjectures the following:
There is a constant $M>0$ such that if $K \subset \mathbb R ^n$ is an embedded torus with one boundary component
and $\partial K$ is a unit circle, then there is a closed curve of length $\ell$ in $K$ which is not
null-homotopic and satisfies
$\ell^2\leq M(area \,K-\pi)$.

Note that the area of a disk bounding the unit circle is $\pi $ so by a surgery one can get
a torus with boundary $K$ with area arbitrarily close to $\pi $. What the conjecture really says is that if the area is close to
$\pi $ then necessarily there is a `short' non null homotopic curve in $K$. Clearly if the area is much bigger than $\pi $, say $2\pi $, then the conjecture follows from the classical result of Loewner \cite {Pu}. 

The purpose of this note is to show that this conjecture holds. In fact we show a slightly stronger result namely that the inequality holds 
for the length of a non-separating simple closed curve in $K$. We show further
 that this result also holds for any orientable surface $S$ with
a single boundary component equal to the unit circle. 

I would like to thank S. Sabourau for suggesting that my proof applies to higher genus surfaces as well.

\section{Length area inequality and systoles}

We will use the co-area formula \cite[ Theorem 13.4.2]{BZ}, which we state now in a simplified form:

\begin{Co-area} \label{coarea}Let $M$ be a Riemannian 2-manifold and let $f:M\to \mathbb R$ be a 1-Lipschitz function. Then
$$\rm{area}\, (M)\geq \int \rm {length}\, (f^{-1}(t))dt .$$

\end{Co-area}

We recall also Loewner's inequality

\begin{Loewner} \label{Loewner}Let $T$ be a Riemannian 2-torus. Then $T$ has a non-contractible geodesic $\gamma $ of length $\ell$ satisfying
$\ell ^2\leq \dfrac {2}{\sqrt 3} \rm{area } T$.

\end{Loewner}

We state now our main result:

\begin{Thm} Let $T$ be a Riemannian torus with a single boundary component embedded in $\mathbb R ^n$. Assume that $\partial T$ is isometric to the
unit circle. Then there is a non separating simple closed curve in $T$ of length $\ell$ such that $$\ell^2\leq 10^3(\rm{area} \,T-\pi).$$

\end{Thm}
\begin{proof}
Let $\ell$ be the length of the shortest non separating simple closed curve in $T$. 

Without loss of generality we may assume that $\partial T$ lies on the $xy$-plane. We may further assume that the functions
$X:(x_1,...,x_n)\to x_1$ and $Y:(x_1,...,x_n)\to x_2$ are Morse functions for $T$. Indeed a slight linear perturbation of $X,Y$ gives Morse functions \cite[p.43]{GP}, 
and this slight perturbation won't affect significantly the calculations that follow. Alternatively this can be obtained by slightly deforming $S$.

We note that if $X^{-1}(t)$ contains a non separating loop then 

\noindent $\rm{length}\,(X^{-1}(t))\geq \ell$. Note also that by Morse Theory if for some $a<b$
$X^{-1}(a), X^{-1}(b)$ contain a non separating loop then $X^{-1}(t)$ contains a non separating loop for all $a<t<b$. We remark that
for each $t\in [-1,1]$, $X^{-1}(t)$ contains a simple arc $\al _t$ spanning $X^{-1}(t)\cap \partial T$. Clearly $\al _t$ has length greater than the
corresponding geodesic $\gamma _t$ joining the same points. Assume now that for some $a<b $ with $b-a>\ell/10$ $X^{-1}(a), X^{-1}(b)$ contain a non separating loop. Then by the length-area
inequality 
$$\rm{area}\, (T)\geq \int \rm {length}\, (\al _t)dt +\ell^2/10\geq \pi +\ell^2/10 .$$
It follows that in this case the theorem holds as $\ell^2<10^3(\ell^2/10)$.

Similarly if the set of $t$ for which $$\rm {length}\, (\al _t)\geq \rm {length}\, (\gamma _t)+\dfrac {\ell}{10}$$
has measure greater or equal to $\ell/100$ then
$$\rm{area}\, (T)\geq \int \rm {length}\, (\al _t))dt \geq \pi+\ell^2/1000 $$ and the theorem holds.

Let $[a_1,b_1], [c_1,d_1] $ be maximal intervals with the property that

 $X^{-1}(a_1), X^{-1}(b_1)$ contain a non separating loop and $Y^{-1}(c_1), Y^{-1}(d_1)$ contain a non separating loop. Clearly $b_1-a_1\leq \ell/10, d_1-c_1\leq \ell/10$. For each $t\in [-1,1]$ denote by $\beta _t$  the simple arc in $Y^{-1}(t)$ spanning $Y^{-1}(t)\cap \partial T$ and by $\gamma '_t$ the geodesic arc joining the same points.

By the previous argument there are $a,b,c,d$ with
$0\leq a_1-a\leq \ell/100,0\leq b- b_1\leq \ell/100,0\leq c_1-c\leq \ell/100,0\leq d-d_1\leq \ell/100$ such that for $z\in \{a,b\}$
$$\rm {length}\, (\al _z)-\rm {length}\, (\gamma _z)\leq \ell/10$$ and for $z\in \{c,d\}$
$$\rm {length}\, (\beta _z)-\rm {length}\, (\gamma ' _z)\leq \ell/ 10 .$$

We consider now the union of arcs:

$\al _a, \al _b$ restricted to $[c,d]$ and $\beta _c, \beta _d$ restricted to $[a,b]$. This union is a separating simple closed
curve $w$ on $T$. Let's denote by $T_1$ the connected component of  $T\setminus w$ containing $\partial T$ and by $T_2$ the other connected component of  $T\setminus w$. We remark that $w$ has the following properties:

1. ${\rm length} (w)\leq \dfrac {8\ell}{10}+\dfrac {8\ell}{100}<\dfrac {9\ell}{10}$.

2. The shortest non separating loop in $T$ is homotopic to a loop contained in $T_2$. Indeed every loop in $T_1$ is separating
and any arc with endpoints on $w$ is homotopic to a subarc of $w$. 

We note that already property 1 above suffices to answer Young's original question if we interpret $\ell$ in the previous part of
the proof to be the length of the shortest non-null homotopic loop.

By the isoperimetric inequality 

$$\rm {area} (T_1)\geq \pi -\dfrac{1}{4\pi }\left (\dfrac {9\ell}{10}\right )^2 \ \ \  \ \ \ \ (*).$$

We fill $w$ by a disk $D$ of arbitrarily small area to obtain a torus $T'=D\cup T_2$ of area less or equal to 
$\rm {area} (T_2)+\epsilon$ for some arbitrarily small $\epsilon >0$.  By choosing carefully the metric on the gluing we can make sure that $T'$ is a smooth riemannian manifold.

If $\ell_1$ is the length of the shortest non separating loop on $T'$ since $\partial T_2$ has length less than
$9\ell/10$ we have that $\ell\leq \ell_1+\dfrac {9\ell}{20}$ so $\ell\leq 2\ell_1$.
Applying Loewner's inequality to $T'$ we have that
$$\ell_1^2\leq \dfrac{2}{\sqrt 3}(\rm {area} (T_2)+\epsilon )$$
and since this holds for any $\epsilon >0$ we obtain
$$\rm {area} (T_2)\geq \dfrac{\sqrt 3}{2}\ell_1^2 \Rightarrow \rm {area} (T_2)>\dfrac{\ell^2}{8 } .$$
Combining this with (*) we have
$$\rm {area} (T)=\rm {area} (T_1)+\rm {area} (T_2) >\pi+\dfrac {1}{100}\ell^2.$$
This inequality clearly implies the theorem.

%

%
%
%
%
%
%
%
%
%

\end{proof}

\begin{Thm} Let $S$ be a Riemannian surface with a single boundary component embedded in $\mathbb R ^n$. Assume that $\partial S$ is isometric to the
unit circle. Then there is a non separating loop in $S$ of length $\ell$ such that $$\ell^2\leq C(\rm{area} \,S-\pi)$$ where $C$ is a universal constant that does not
depend on $S$.

\end{Thm}
\begin{proof} The argument of the previous theorem applies with little change. Let $\ell$ be the length of the shortest non separating simple closed curve in $S$. 
If $X:(x_1,...,x_n)\to x_1$ is a Morse function for $S$
(where $\partial S$ lies on the $xy$-plane) then in this case there are $a_1<b_1<a_2<b_2<...<a_n<b_n$ such that all
non-separating loops of $S$ lie in $X^{-1}([a_1,b_1]\cup...\cup [a_n,b_n])$ where $\sum (b_i-a_i)\leq \ell/10$.
Using in a similar way the Morse function $Y:(x_1,...,x_n)\to x_2$ and arguing as in the previous theorem we arrive at a collection
of separating simple closed curves $w_1,...,w_k$ each of which has length less than $\ell $. If we denote by $S_i$ the connected
component of $S\setminus w_i$ that does not contain $\partial S$ then we may apply Gromov's generalization of Loewner's inequality for $S_i$
(see \cite[ sec. 2.C]{Gr2}, \cite[Cor. 5.2.B]{Gr1}) to obtain the desired bound on $\ell $.
\end{proof}

\section{Discussion}

Theorems 2.1, 2.2 `quantify' the defect of a filling of $S^1$ by a general surface rather than a disc. They say that if the filling
by a surface is `close' to the optimal filling by a disc then the surface is `close' to a disc as it has a short non-separating geodesic.
This has a similar flavor to the classical Bonnesen inequality \cite{Bo} on isoperimetric defect quantifying how far is a region from being optimal
for the isoperimetric inequality. A strengthening of Loewner's inequality in this spirit is given in \cite{HKK}.
We note also that Babenko in \cite{Ba} has studied systoles of manifolds with boundary. 

A related well known question to theorems 2.1, 2.2 is the conjecture of Gromov on the filling
volume (area) of $S^1$ \cite[sec.5.5, p.60]{Gr1}. One wonders if the analog of theorem 2.1 holds in this case, namely whether
if $T$ is a torus with boundary filling $S^1$ then there is a non-separating curve of length $\ell $ in $T$ satisfying
$$\ell^2\leq C(\rm{area} \,T-2\pi)$$ for some universal constant $C$. Note that the Gromov's conjecture is known to hold for tori
with boundary \cite{BCIK} but is still open for higher genus surfaces.

\end{document}
\bye